\documentclass[12pt,a4paper,leqno]{article}
\usepackage[T1]{fontenc}
\usepackage[latin1]{inputenc}
\usepackage{float}
\usepackage{amsfonts}
\usepackage{amssymb}
\usepackage{amsthm}
\usepackage{graphicx}
\catcode`\é=\active \def é{\' e}

\newfont{\Blackboard}{msbm10 scaled 1200}

\newfont{\roma}{cmr10 scaled 1200}

\def \Z {{\mathbb{Z}}}
\def \R {{\mathbb{R}}}
\def \N {{\mathbb{N}}}
\def \C {{\mathbb{C}}}

\def \eps {{\epsilon}}

\def \om {{\omega}}

\newcommand{\dint}{\displaystyle\int}
\newcommand {\nc}   {\newcommand}

\newcommand{\norm}[2]{\|#1 \| _{#2} }

\newcommand{\go}[1]{O \left( \dfrac{1}{\lb ^{#1}} \right)}

\newcommand{\gok}[1]{O \left( \dfrac{1}{k ^{#1}} \right)}
\newcommand{\pok}[1]{o \left( \dfrac{1}{k ^{#1}} \right)}
\newcommand{\goko}{O ( \dfrac{1}{k})}
\newcommand{\poko}{o \left( \dfrac{1}{k} \right)}

\nc {\be}   {\begin{equation}} \nc {\ee}   {\end{equation}} 
\nc {\beq}  {\begin{eqnarray}} \nc {\eeq}  {\end{eqnarray}}
\nc {\beqs} {\begin{eqnarray*}} \nc {\eeqs} {\end{eqnarray*}}

\newcommand{\bthe}{\begin{theorem}}
\newcommand{\ethe}{\end{theorem}}
\newcommand{\blem}{\begin{lemma}}
\newcommand{\elem}{\end{lemma}}
\newcommand{\bprop}{\begin{proposition}}
\newcommand{\eprop}{\end{proposition}}
\newcommand{\brk}{\begin{remark}}
\newcommand{\erk}{\end{remark}}
\newcommand{\bco}{\begin{corollary}}
\newcommand{\eco}{\end{corollary}}
\newcommand{\bdef}{\begin{definition}}
\newcommand{\fdef}{\end{definition}}
\newcommand{\bex}{\begin{example}}
\newcommand{\eex}{\end{example}}

\newcommand{\caA}{{\cal A}}

\newcommand{\caH}{{\cal H}}

\newcommand{\lb}{\lambda}

\newcommand{\al}{\alpha}
\newcommand{\dsum}{\displaystyle \sum}

\newtheorem{theorem}{Theorem}[section]
\newtheorem{lemma}[theorem]{Lemma}
\newtheorem{corollary}[theorem]{Corollary}
\newtheorem{remark}[theorem]{Remark}
\newtheorem{definition}[theorem]{Definition}
\newtheorem{proposition}[theorem]{Proposition}
\newtheorem{example}[theorem]{Example}

\def\dfrac{\displaystyle \frac }

\def\eps{\varepsilon }

\def\rr{\mathbb R}

\def\cc{\mathbb C}
\def\HH{\mathcal H}
\def\AA{\mathcal A}

\def\dint{{\displaystyle\int}}
\def\lb {{\lambda}}
\def\eps{{\epsilon}}

\title{Non uniform stability for the Timoshenko beam with tip load}

\author{Denis Mercier, Virginie R\'egnier \thanks{Laboratoire de Mathématiques
et ses Applications de Valenciennes, FR CNRS 2956, Institut des Sciences et Techniques de
Valenciennes, Université de Valenciennes et du Hainaut-Cambrésis, Le Mont
Houy, 59313 VALENCIENNES Cedex 9, FRANCE, e-mails : denis.mercier@univ-valenciennes.fr ; virginie.regnier@univ-valenciennes.fr}}

\def\edc{\end{document}}
\begin{document}
\maketitle 


\begin{abstract}
In this paper we consider a hybrid elastic model consisting of a Timoshenko beam and
a tip load at the free end of the beam. Under the equal speed wave propagation condition, we show polynomial decay for the model
which includes the rotary inertia of the tip load when feedback boundary
moment and force controls are applied at the point of contact between the beam and
the tip load.

\end{abstract}

\maketitle
%WWWWWWWWWWWWWWWWWWWWWWWWWWWWWWWWWWWWWWWWWWWWWWWWWWWWWWWWWWWWWWWWWWWWWWWWWWWWWWWWWWWWWW
%WWWWWWWWWWWWWWWWWWWWWWWWWWWWWWWWWWWWWWWWWWWWWWWWWWWWWWWWWWWWWWWWWWWWWWWWWWWWWWWWWWW
\section{Introduction}

Beam structures have been studied extensively in the last decades: Euler-Bernoulli, Rayleigh and Timoshenko beams. The latest model is more accurate since it takes into account not only the rotary initial energy but also its deformation due to shear (see Timoshenko's book for physical explanations: \cite{Timobook}). A non-exhaustive list of contributions is: \cite{amtu}, \cite{maya}, \cite{caszua}, \cite{feng}, \cite{he}, \cite{kim}, \cite{liu}, \cite{mer}, \cite{merreg}, \cite{mor}, \cite{vu}, \cite{xu}. \\
\\
In this paper, we study the stabilization of a Timoshenko beam which has a tip load attached to one free end. The beam is clamped at one end while the tip load is fixed to the other end $x=1$ in such a manner that the center of mass of the load is coincident with its point of attachment to the beam. We assume interaction between the beam and the load. Thus the forces and moments within the vibrating beam are transmitted to the tip load which moves in accordance with Newton's law. Dissipation is introduced into the coupled model by applying feedback boundary moment and force controls on the displacement and shear  velocities. Multiplying the initial equations by suitable constants and rescaling in time, the coupled motions of the beam-load structure are governed by the following problem :

\begin{eqnarray}
\label{e11} (u_{tt}-(u_x+y)_x)(x,t) &=& 0,  \hskip 2cm \mbox{for } (x,t) \in (0,1) \times (0,\infty), \\
\label{e112}(y_{tt}-ay_{xx}+b(u_x+y))(x,t)&=&0, \hskip 2cm \mbox{for } (x,t) \in (0,1) \times (0,\infty), \\
u(0,t)=y(0,t)&=&0, \hskip 1.90cm \mbox{ for } t \in (0,\infty),
\label{e13}
\end{eqnarray}

with the initial conditions

\begin{equation}
\label{e14} u(x,0)=u_0(x), \ u_t(x,0)=u_1(x), \ y(x,0)=y_0(x), \ y_t(x,0)=y_1(x), \ \ \ \ \mbox{ for} \ \ x \in (0,1),
\end{equation}
and the boundary dissipation law

\begin{eqnarray}
u_{tt}(1,t)+k_1(u_x(1,t)+y(1,t))=-k_2u_t(1,t), \ \ \ \  \quad{\rm for} \ \ t \in (0,\infty),\label{e15}\\
y_{tt}(1,t)+k_3 y_x(1,t)=-k_4 y_t(1,t), \ \ \ \  \quad{\rm for} \ \ t \in (0,\infty),\label{e16}
\end{eqnarray}

where $a,b,k_1,k_2,k_3,k_4$ are strictly positive constants. \\
Denote by $\rho$, $I_{\rho}$, $EI$, $\kappa$, $\omega(x,t)$ and $\varphi(x,t)$, the mass density, the moment of mass inertia, the rigidity coefficient, the shear modulus of the elastic beam, the lateral displacement at location $x$ and time $t$ and the bending angle at location $x$ and time $t$ respectively. Then, our model coincides with those of \cite{feng}, \cite{grobb}, \cite{he}, \cite{vu}, ... with $u(x,t)= \omega \left( x,\sqrt{\dfrac{\kappa}{\rho}} t \right)$ , $y(x,t)= - \varphi \left( x,\sqrt{\dfrac{\kappa}{\rho}} t \right)$, $a =\dfrac{(EI) \rho}{\kappa I_{\rho}}$ and $b= \dfrac{\rho}{I_{\rho}}$. \\
\\
\\
This system is studied by Kim and Renardy (\cite{kim}), but with other boundary dissipation laws and it is then proved to be exponentially stable. \\
M. Bassam, D. Mercier, S. Nicaise and A. Wehbe also consider the same system but with other boundary dissipation laws. They study the decay rate of the energy of the Timoshenko beam with one boundary control
acting in the rotation-angle equation. Under the equal speed wave propagation condition ($a = 1$) and if $b$ is outside a discrete set of exceptional values, using a spectral analysis, the authors prove non-uniform stability and obtain the optimal polynomial energy decay rate. On the other hand, if $\sqrt{a}$ is a rational number and if $b$ is outside another
discrete set of exceptional values, they also show a polynomial-type decay rate using a frequency domain approach. See \cite{maya} and the references therein, particularly papers by F. Alabau-Boussouira (\cite{ala}), J.E. Muñoz Rivera and R. Racke, papers by S.A. Messaoudi and M.I. Mustafa, papers by A. Wehbe and his co-authors: A. Soufyane and W. Youssef...   \\
The stabilization of the Timoshenko beam is a subject of interest for many other authors recently: D. Feng, W. Zhang with a nonlinear feedback control (\cite{feng}), W. He, S. Zhang, S. Ge (see \cite{he}), {\"O}.~Morg{\"u}l with a dynamic boundary control (\cite{mor}). \\
The spectral analysis is studied by M.A. Shubov (\cite{shub} and Q.P. Vu, J.M. Wang, G.Q. Xu, S.P. Yung (\cite{vu}). \\
Systems of Timoshenko beams, serially connected or forming a tree-shaped network are another interesting point: see \cite{hanxu}, \cite{liu}, \cite{xu}, \cite{ZhangXu}. \\
\\
The system we consider is also studied by M. Grobbelaar-Van Dalsen in \cite{grobb} with the same feedback controls as ours. It is proved that uniform stability holds under a condition (called condition Z.)  Unfortunately this condition is not easy to check and the exponential stability (for $a=1$) remains an open question. This is why, in the present work, we consider the same problem which is still open. The main goal of this paper is to prove that the decay of the energy is not exponential, but polynomial. \\
We conjecture that the same results hold in the case $a \neq 1$. The computations are more complicated and still have to be performed. \\
\\
In Section \ref{wellposedness}, the abstract framework is introduced and the operator is proved to be m-dissipative in the energy space. The existence and uniqueness of a solution of the abstract evolution problem in appropriate spaces is established. The energy of the solution is then proved to decay to zero, using Benchimol Theorem (\cite{B}) (i.e. the operator is proved to have no purely imaginary eigenvalues). \\
Section \ref{spectrum} is dedicated to a thorough analysis of the spectrum of both the dissipative operator and the conservative associated operator. In particular, we give asymptotic expansions for the eigenvalues (cf. (\ref{dl1}), (\ref{dl2}), (\ref{dl01}) and (\ref{dl02})). \\
It is proved, in Section \ref{secriesz}, that the system of generalized eigenvectors of the dissipative operator (introduced in the latest section) forms a Riesz basis of the energy space. To this end, we use Theorem 1.2.10 of \cite{abd} which is a rewriting of Guo's version of Bari Theorem with another proof (see \cite{Guo}). The proof requires the asymptotic analysis performed before. \\
At last, the solution is explicitly expressed using the Riesz basis to prove that the energy decays polynomially (see Section \ref{secdecay}). \\
To examplify and validate our results, we give numerical computations and figures representing the spectrum of the dissipative operator in Section \ref{num}.

%%%%%%%%%%%%%%%%%%%%%%%%%%%%%%%%%%%%%%%%%%%%%%%%%%%%%%%%%%%%%%%%%%%%%%%%%%%%%%%%%%%%%%%%%%%%%%%%%%%%%%%%%%%%%%%%%%%%%%%%%%%%%%%%%%%%%%%%%%%%%%  Well-posedness
%%%%%%%%%%%%%%%%%%%%%%%%%%%%%%%%%%%%%%%%%%%%%%%%%%%%%%%%%%%%%%%%%%%%%%%%%%%%%%%%%%%%%%%%%%%%%%%%%%%%%%%%%%%%%%%%%%%%%%%%%%%%%%%%%%%%
\section{ Well-posedness and strong stability} \label{wellposedness}

In this section we study the existence, uniqueness and strong stability of the solution of System (\ref{e11})-(\ref{e16}).
Setting
$$\Omega := (0,1)\ \ \quad{\rm and}\ \  H_L^1(\Omega):= \{f \in H^1(\Omega) : f(0) = 0\},$$
we define the energy space $\mathcal{H}$ as follows
$$\mathcal{H}:= H_L^1(\Omega) \times L^2(\Omega) \times H_L^1(\Omega)\times L^2(\Omega)\times \cc \times \cc,$$
with the inner product defined by
\begin{equation}
\begin{array}{lll}
<U,U_1>_{\HH}& :=& \displaystyle\int_0^1\Big(v\overline{v_1} + b^{-1}z\overline{z_1} + ab^{-1}y_x\overline{y_{1x}} + (u_x + y)
(\overline{u_{1x} + y_1})\Big)(x) dx \\&+& \dfrac{1}{k_1} \eta \overline{\eta_1 } +\dfrac{1}{k_3} \gamma \overline{\gamma_1},
\end{array}
\label{e23}
\end{equation}
for all $U = (u,v,y,z,\eta,\gamma)$, $U_1 = (u_1,v_1,y_1,z_1,\eta_1,\gamma_1)$ $\in$ $\HH$.
\begin{remark}
The norm $<U,U>_{\HH}^{\frac12}$ induced by (\ref{e23}) is equivalent to the usual norm of $\HH$.
\end{remark}

For shortness we denote by $\norm{\cdot}{}$ the  $L^2(\Omega)$-norm.

Now we define the linear unbounded operator $\mathcal{A}:D(\mathcal{A})\rightarrow\HH$ by:
$$\begin{array}{ll}
D(\mathcal{A}):=& \lbrace U= (u,v,y,z,\eta,\gamma) \in \mathcal{H} : u,y \in H^2(\Omega),v\in H_L^1(\Omega),z\in H_L^1(\Omega), \\
&\eta= v(1),\gamma=z(1)  \rbrace,
\\
\end{array}\\
$$
$\forall\, U  \in D(\mathcal{A}),$
\be\label{e25}
\mathcal{A}U := \big(v, (u_x + y)_x, z, ay_{xx} - b(u_x + y),-k_1(u_x(1)+y(1))-k_2 \eta, -k_3 y_x(1)-k_4 \gamma \big).\ee
%  \forall U = (u,v,y,z,\eta,\gamma) \in D(\mathcal{A}).\label{e25}

The associated conservative operator is $\mathcal{A}_0:D(\mathcal{A})\rightarrow\HH$ defined as $\mathcal{A}$ but with $k_2=k_4=0$ i.e. 
\be \label{AandA0} \AA U = \AA_0 U - k_2 \eta e_5 - k_4 \gamma e_6, \ee
where $U= (u,v,y,z,\eta,\gamma) \in D(\AA)$, $e_5:= (0,0,0,0,1,0)$ and $e_6:= (0,0,0,0,0,1)$. \\
System (\ref{e11})-(\ref{e16}) is formally rewritten as the evolution equation
 \begin{equation}
(P) \;\; \left\{\begin{array}{ll}
U_t(t) = \mathcal{A}U(t), t \in (0;+ \infty),\\
U(0) = U_0, \ \ \ \ U_0\in \HH,
\label{P}
\end{array}\right.
\end{equation}
with $U(t) = (u,u_t,y,y_t,u_t(1),y_t(1))$ (note that the notation $U$ is kept for this function of the time $t$).

\begin{proposition}\label{maxi}
The operator $\AA$ is m-dissipative in the energy space $\HH$.
\end{proposition} 

\begin{proof}  We start with the dissipativeness.\\
Let $U = (u,v,y,z,\eta,\gamma)$ $\in$ $D(\AA)$. Using (\ref{e23}) and (\ref{e25}), we obtain : 
$$ 
\begin{array}{lll}
<\AA U, U>_{\HH} &= & \dint_0^1 \Big( (u_x + y)_x\overline{v} + b^{-1}\Big(ay_{xx} - b(u_x + y)\Big)\overline{z}\\
 &+& ab^{-1}z_x\overline{y_x} + (v_x + z)(\overline{u_x + y})\Big)(x) dx\\
 & + &\dfrac{1}{k_1}(-k_1(u_x(1)+y(1)-k_2 \eta)\overline{\eta} +\dfrac{1}{k_3} (-k_3 y_x(1)-k_4 \gamma) \overline{\gamma} .
\end{array} $$

Then, integrating by parts and using the boundary conditions, we get

\begin{equation}
\Re <\AA U, U>_{\HH} = -\dfrac{k_2}{k_1} \vert v(1)\vert^2-\dfrac{k_4}{k_3} \vert z(1)\vert^2  \leq 0.\label{e27}
\end{equation}
Therefore, $\AA$ is dissipative.  

Now, we prove that $\AA$ is maximal. For that purpose, we consider any $f = (f_1,f_2,f_3,f_4,\eta_1,\gamma_1)$ $\in$ $\HH$ and 
we look for a unique element $U = (u,v,y,z,\eta,\gamma)$ $\in$ $D(\AA)$ such that
$$\AA U = f.$$ 
Equivalently, we get $v=f_1$, $z=f_3, \eta=f_1(1), \gamma=f_3(1),$ and we have the following system to solve: 
\begin{equation}
(u_x + y)_x(x) = f_2(x),\label{e28}
\end{equation}
\begin{equation}
(ay_{xx} - b(u_x + y))(x) = f_4(x),\label{e29}
\end{equation}
\begin{equation}
-k_1(u_x(1)+y(1))-k_2 \eta=\eta_1,\label{b1}
\end{equation}
\begin{equation}
-k_3 y_x(1)-k_4 \gamma =\gamma_1.\label{b2}
\end{equation}

From (\ref{e28}) it follows $u_x(x)+y(x)=F_2(x)+a_1,$ where $F_2(x)=\int_0^x f_2(u)du$ and $a_1$ is a constant. Consequently (\ref{e29}) becomes 
\be \label{new1} a y_{xx}(x)= f_4(x)+b F_2(x)+b a_1. \ee 

Let $G_4$ (resp. $G_2$) be the unique solution of $(G_4)_{xx}=f_4$  (resp.   $(G_2)_{xx}=F_2$) satisfying
$G_4(0)=(G_4)_x(0)=0$ (resp. $G_2(0)=(G_2)_x(0)=0$). Then, we find that the solutions of (\ref{new1}) satisfying  $y(0)=0$ are 
$$y(x)=\dfrac{1}{a} G_4(x)+\dfrac{b}{a} G_2(x)+\dfrac{b}{a}a_1 \dfrac{x^2}{2} + a_2 x,$$ where $a_2$ is a constant. 
Now, let $y$ as previously. Clearly $y \in  H^2(\Omega) \cap H^1_L(\Omega)$ and we find that necessarily 
$u(x)=\dint_0^x (-y(u)+F_2(u)) du +a_1 x$ (since $u(0)=0$).

Inserting $u_x(1)+y(1)= F_2(1)+a_1$ in (\ref{b1}) we get an equation with only the unknown $a_1$ and this equation admits a unique solution. 
Therefore (\ref{b2}) becomes an equation with a unique solution $a_2.$ Finally, inserting these two constants in $u$ and $y,$ 
it is easy to check that we have found a unique $U=(u,v,y,z,\eta,\gamma)\in  D(\AA)$ such that $\AA U=f.$ 
  
Therefore we deduce that $0\in \rho(\mathcal{A})$. Then, by the resolvent identity, for $\lambda >0$ small enough, $R(\lambda I-\AA)=\HH$ (see Theorem 1.2.4 in \cite{LZ}). \end{proof} 

Due to Lumer-Phillips Theorem (see \cite{Pazy}, Theorem 1.4.3), it follows from Proposition \ref{maxi} that the operator $\AA$ generates a $C_0$-semigroup of contractions $e^{t\AA}$ on $\HH$. Consequently it holds:

\bthe\label{th2.3} (Existence and uniqueness)\\
(1) If $U_0$ $\in$ $\HH$, then System $(P)$ has a unique solution\\
$$U \in C^0(\rr_+,\HH).$$

(2) If $U_0$ $\in$ $D(\AA)$, then system $(P)$ has a unique solution
$$U  \in C^0(\rr_+,D(\AA))\cap C^1(\rr_+,\HH).$$
\ethe

\brk \label{A0} Let $(P_0)$ be the conservative problem associated to problem $(P)$ (in other words $(P_0)$ is Problem $(P)$ with $k_2=k_4=0$) and $\AA_0$ be the associated operator then Proposition \ref{maxi} (resp. Theorem \ref{th2.3}) remains true for $\AA_0$ (resp. $(P_0)\;$).   
\erk

To end this section we give a first stability result: 

\bthe\label{th21}(Strong stability)\\
System (\ref{e11})-(\ref{e16}) is strongly stable, i.e for any solution $U$ of $(P)$ with initial data $U_0 \in\HH$, it holds 
$$\lim_{t\rightarrow\infty }E(t)=0,$$ where $E(t)=\dfrac{1}{2}\|U(t)\|^2_\HH.$ 
\ethe

\begin{proof}
%\noindent{\bf proof of theorem 2.3}\\
Since the resolvent of $\AA$ is compact in $\HH$, using Benchimol Theorem \cite{B}, System $(P)$ is strongly stable if and only if
$\AA$ does not have purely imaginary eigenvalues.  We have already seen that $\AA$ is invertible. Thus we consider $\lb \in \rr ^*$ and $U=(u,v,y,z,\eta,\gamma) \in D(\AA)$ such that 

$$\AA U = i \lb U.$$ 

Since $\Re <\AA U,U> =0,$ we get from (\ref{e27}) that $\eta=v(1)=0$ and $\gamma=z(1)=0,$ and we deduce that $(u,v)$ satisfies

\begin{equation}
\left\{ \begin{array}{lll}
(u_{xx}+y_x+\lambda^2 u)(x)&=&0, \\
(a y_{xx}- bu_{x}-b y+\lambda^2 y)(x)&=&0, \\
\end{array}\right.\label{e210}
\end{equation}

with the boundary conditions 

\begin{equation}
\left\{ \begin{array}{lll}
u(0)&=y(0)&=0, \\
u(1)&=y(1)&=0, \\
u_x(1)&=y_x(1)&=0. \\
\end{array}\right.\label{e211}
\end{equation}

From the first equation of (\ref{e210}), $y_x(x) = -u_{xx}(x)-\lambda^2 u(x)$. Thus $a y_{xx}(x) = - a u^{(3)}_x(x) - a \lambda^2 u_x(x)$. Now, from the second equation of (\ref{e210}), it follows: $a y_{xx} = bu_{x}(x) + b y(x) - \lambda^2 y(x)$. Then $u$ is solution of 

\be\label{ord1} a u_x^{(4)}(x) + (a + 1) \lb^2 u_x^{(2)}(x) + (\lb^2 -b) \lb^2 u(x) = 0.\ee

Note that, from the boundary conditions (\ref{e211}) and the relations (\ref{e210}), it also holds $u_{xx}(1)=u_{xxx}(1)=0.$
Thus $u$ is solution of (\ref{ord1}) and satisfies $u(1)=u_{x}(1)=u_{xx}(1)=u_{xxx}(1)=0.$ Therefore, from the general theory of ordinary differential equations, we deduce that $u \equiv 0.$ \\
It follows that $y \equiv 0$ and finally $U \equiv 0.$ Consequently, $\AA$ has no eigenvalue on the imaginary axis. 
\end{proof}

%%%%%%%%%%%%%%%%%%%%%%%%%%%%%%%%%%%%%%%%%%%%%%%%%%%%%%%%%%%%%%%%%%%%%%%%%%%%%%%%%%%%%%%%%%%%%%%%%%%%%%%%%%%%%%%%%%%%%%%%%%%%%%%%%%%%%%%%%%%%%%   Spectrum analysis
%%%%%%%%%%%%%%%%%%%%%%%%%%%%%%%%%%%%%%%%%%%%%%%%%%%%%%%%%%%%%%%%%%%%%%%%%%%%%%%%%%%%%%%%%%%%%%%%%%%%%%%%%%%%%%%%%%%%%%%%%%%%%%%%%%%%

\section{Spectrum analysis for the case $a = 1$} 
\label{spectrum}

\subsection{Main results and notation} 

Let us begin with announcing the main results concerning the spectrum analysis. The following theorem is also a way to introduce the notation which is used during the whole section. That is why it is given first whereas establishing its proof is the goal of the following subsections.

\begin{theorem} \label{ppr} (Spectrum and eigenvectors of both the conservative and dissipative operators) \\

   \begin{enumerate} 
   \item {\bf Spectrum of $\AA_0$.}\\
Let $\sigma_0$ be the spectrum of $\AA_0.$ We can split $\sigma_0$ as follows: 
$$\sigma_0=\sigma_0^1\cup \sigma_0^2,$$ where  
$$\sigma_0^1=\{\kappa_{i}^0\}_{i \in I_0},$$ 
and $I_0$ is a finite set, the multiplicity of $\kappa_{i}^0$ is $m_{i,0}$ and is finite.
$$\sigma_0^2=\{\lb_k^{j,0}\}_{j=1,2, |k|\geq k_0},$$
and the multiplicity of $\lb_k^{j,0} (j=1,2)$ is one.\\

   \item {\bf Eigenvectors of $\AA_0$.}\\
For each $i \in I_0$,  we will denote by ${\tilde{\phi_i^l}},l=0,...,m_i-1,$ a system of independent eigenvectors  associated with $\kappa_i^0 \in \sigma_0^1.$ \\
For each $k\in \Z, |k|\geq k_0,$  we will denote by $\phi_k^j  (j=1,2)$ an associated 
eigenvector of $\lb_k^{j,0} (j=1,2) \in \sigma_0^2.$ \\
Moreover, since $\AA_0$ is skew-adjoint, the system 
$${\cal F}_0= \{\tilde{\phi_i^l}\}_{i\in I_0,l=0,...,m_i-1} \cup \{\phi_k^j\}_{|k|\geq k_0,j=1,2}$$
can be chosen such that ${\cal F}_0$ forms an orthonormal basis of $\caH.$ \\

   \item {\bf Spectrum of $\AA$.}\\
Similarly, let $\sigma$ be the spectrum of $\AA.$ We can split $\sigma$ as follows: 
$$\sigma=\sigma^1\cup \sigma^2,$$ where  
$$\sigma^1=\{\kappa_{i}\}_{i\in I},$$ 
and $I$ is a finite set, the algebraic multiplicity of  $\kappa_{i}$ is $m_{i}$ and is finite, the geometric multiplicity is $n_i,$ with $1\leq n_i\leq m_i$.
$$\sigma^2=\{\lb_{k}^j\}_{j=1,2, |k|\geq k_0},$$
and the multiplicity of $\lb_{k}^j (j=1,2)$ is one. \\

   \item {\bf Generalized eigenvectors of $\AA$.}\\
For each $i \in I$,  we will denote by $\{\tilde{\psi_{ik}^l}\}_{l=1}^{\delta_{ik}},k=1...,n_i,$ a system of independent generalized eigenvectors  associated with $\kappa_i \in \sigma^1,$ which forms Jordan chains, i.e 
$\delta_{ik}\geq 1,$\\$ k=1,...,n_i, \sum_{k=1}^{n_i}\delta_{ik} = m_i,$ $$(\AA-\kappa_i I)\psi_{ik}^l=\psi_{ik}^{l-1}, l=1,...,\delta_{ik}, $$ where we assume that $\psi_{ik}^{0}=0$.\\
For each $k\in \Z, |k|\geq k_0,$  we will denote by $\psi_k^j ( j=1,2)$ an associated 
eigenvector of $\lb_k^{j} (j=1,2) \in \sigma^2.$ \\
The system 
$${\cal F}= \{\tilde{\psi_i^l}\}_{i\in I,l=0,...,m_i-1} \cup \{\psi_k^j\}_{|k|\geq k_0,j=1,2}$$
is chosen such that any $\psi_k^j \in {\cal F}, |k|\geq k_0, j=1,2,$ satisfies $\|\psi_k^j\|_{\caH}=1.$
   \end{enumerate}

\end{theorem}

\subsection{Eigenvalues of $\AA.$}

Let $\lb \in \cc^*$ and $U\neq0, U=(u,v,y,z,\eta,\gamma) \in D(\AA)$ such that  
\be \label {eig1} \AA U= \lb U.\ee

Then $\eta=v(1),\gamma=z(1)$ and $(u,v,y,z)$ is solution of 
$$
\left\{
\begin{array}{l}
v(x)=\lb u(x), x \in (0;1), \\
u_{xx}(x)+y_x(x)=\lambda v(x), x \in (0;1), \\
z(x)=\lambda y(x), x \in (0;1), \\
y_{xx}(x)-bu_x(x)-by(x)=\lambda z(x), x \in (0;1), \\
u(0)=0,\\
y(0)=0,\\
\lb^2u(1)+k_1(u_x(1)+y(1))+k_2\lb u(1)=0,\\
\lb^2y(1)+k_3y_x(1)+k_4\lb y(1)=0.\\
\end{array}
\right. 
$$
Eliminating $v$ and $z$ implies that solving (\ref{eig1}) is equivalent to solving:  

\be \label{syseig}
\left\{
\begin{array}{ll}
(i) &(u_{xx}+y_x-\lambda ^2 u)(x)=0, x \in (0;1),  \\
(ii) &(y_{xx}-bu_x-by -\lambda ^2 y)(x)=0, x \in (0;1), \\
(iii) &u(0)=0,\\
(iv) &y(0)=0,\\
(v)&\lb^2u(1)+k_1(u_x(1)+y(1))+k_2\lb u(1)=0,\\
(vi) &\lb^2y(1)+k_3y_x(1)+k_4\lb y(1)=0.\\
\end{array}
\right. 
\ee

From $(i)$ and $(ii)$, it follows that $u$ is solution of 

\be u_x^{(4)}(x) - 2 \lb^2 u_x^{(2)}(x) + (\lb^2 + b) \lb^2 u(x) = 0
\ee

(cf. (\ref{ord1}) with $a=1$ and $\lb$ replaced by $(-i \lb)$). \\
\\
Denoting by $t_1$, $t_2$, $t_3$ and $t_4$ the solutions of the characteristic equation $r^4 - 2 \lb^2 r^2 + \lb^2 (\lb^2 + b)=0$, i.e.

\be\label{ti} t_1(\lb)=t_1=\sqrt{\lb} \sqrt{i \sqrt{b}+\lb}, t_2=-t_1, t_3(\lb)=t_3=\sqrt{\lb} \sqrt{-i \sqrt{b}+\lb}, t_4=-t_3,\ee the general solution of $(i)$ and $(ii)$ is proved to be given by 
\be \label{fond1}u(x)=\dsum_{i=1}^{4} c_i e^{t_i x },\;\;y(x)=\dsum_{i=1}^{4} c_i d_i e^{t_i x }, \ee
where $c_i \in \C,i=1,...4$ and 
\be \label{di} d_1= \dfrac{\lb^2-t_1^2}{t_1},d_2=\dfrac{-\lb^2+t_1^2}{t_1},d_3 =\dfrac{\lb^2-t_3^2}{t_3},d_4=\dfrac{-\lb^2+t_3^2}{t_3}.\ee 
The values for $d_i, i = 1, \ldots, 4$ come from $(i)$, using the expression for $u$ given by $(\ref{fond1})$. \\ 
\\
Note that (\ref{eig1}) and (\ref{e27}) imply $\Re(\lb) \leq 0$. In the proof of Theorem \ref{th21}, the absence of purely imaginary eigenvalues is proved. Thus $\Re(\lb) < 0$ and $t_1$ does not vanish nor $t_3$. The coefficients $d_1$, $d_2$, $d_3$ and $d_4$ are well defined. \\
\\
Therefore the boundary conditions $(iii)- (vi)$ are equivalent  to the system 
$$\left(\begin{array}{llll}
1 &1 &1&1 \\ 
g_{1}(t_1)&g_{1}(t_2)&g_{1}(t_3)&g_{1}(t_4)\\
\lb^2 e^{t_1}g_{2}(t_1)&\lb^2 e^{t_2}g_{2}(t_2)&\lb^2 e^{t_3}g_{2}(t_3)&\lb^2 e^{t_4}g_{2}(t_4)\\
\lb^2 e^{t_1}g_{3}(t_1)&\lb^2 e^{t_2}g_{3}(t_2)&\lb^2 e^{t_3}g_{3}(t_3)&\lb^2 e^{t_4}g_{3}(t_4)\\
\end{array} \right) \left ( \begin{array}{l}
c_1\\
c_2\\
c_3\\
c_4
\end{array}
\right)=0,
$$
where  
\beq \label{g1} g_1(t)&=&-t+\dfrac{\lambda ^2}{t},\\
\label{g2}  g_2(t)&=& \dfrac{k_2 t+(k_1+t) \lambda }{\lb t}, \\
\label{g3} g_3(t)&=&\dfrac{\left(-t^2+\lambda ^2\right) (k_3 t+\lambda  (k_4+\lambda ))}{  \lb^2 t}.
\eeq
Multiplying the third and fourth lines of the previous system by $\dfrac{1}{\lb^2},$ this one is equivalent to   
\be \label{sys1} 
\left(\begin{array}{llll}
1 &1 &1&1 \\ 
g_{1}(t_1)&g_{1}(t_2)&g_{1}(t_3)&g_{1}(t_4)\\
 e^{t_1}g_{2}(t_1)& e^{t_2}g_{2}(t_2)&e^{t_3}g_{2}(t_3)& e^{t_4}g_{2}(t_4)\\
 e^{t_1}g_{3}(t_1)& e^{t_2}g_{3}(t_2)& e^{t_3}g_{3}(t_3)& e^{t_4}g_{3}(t_4)\\
\end{array} \right) \left ( \begin{array}{l}
c_1\\
c_2\\
c_3\\
c_4
\end{array}
\right)=0.
\ee 
Let $M(\lb)$ be the matrix of the previous system and $C=(c_1,c_2,c_3,c_4)^t,$ then we deduce that 
$\lb\in\C\; ( \Re(\lb)<0)$ is an eigenvalue of $\AA$ if and only if $\lb$ is solution  of the characteristic equation 

\be\label{char} \det(M(\lb)) = 0 \Leftrightarrow f(\lb)= 0, \mbox{with} \quad f(\lb) := - \dfrac{1}{16b} \det(M(\lb)).\ee 

(The division by ($-16b$) simplifies the expressions calculated in next subsection for the asymptotic analysis.) \\ 
If $\lb$ is an eigenvalue of $\AA,$ an associated eigenvector has the form 

$$U=(u, \lb u, y ,\lb y,\lb u(1), \lb y(1)),$$ 

and is given by $C$ a nontrivial solution of (\ref{sys1}) and formulas (\ref{fond1})-(\ref{di}). Moreover the geometric multiplicity of $\lb$ is equal to the dimension of the kernel of $M(\lb).$ \\
\\
Note that the expressions of $g_2$ and $g_3$ depend on the values of $k_2$ and $k_4$. Thus the eigenvalues and eigenvectors of $\AA_0$ are different from those of $\AA$.

%%%%%%%%%%%%%%%%%%%%%%%%%%%%%%%%%%%%%%%%%%%%%%%%%%%%%%%%%%%%%%%%%%%%%%%%%%%%%%%   Asymptotic analysis
%%%%%%%%%%%%%%%%%%%%%%%%%%%%%%%%%%%%%%%%%%%%%%%%%%%%%%%%%%%%%%%%

\subsection{Asymptotic analysis}

In this part we study the asymptotic behaviour of the large eigenvalues which are proved to lie in the strip 
$${\cal B}=\lbrace\lb \in \C: -\al \leq \Re(\lb)<0\rbrace,$$  where $\al>0$ is fixed and chosen large enough. \\
The large eigenvalues are also proved to be simple and the asymptotic expansions $(\ref{dl1})$ and $(\ref{dl2})$ are established. 

We first start by:  
\blem \label{asympexp0}(Asymptotic behaviour of the characteristic equation)\\
There exists $\al>0$ such that the eigenvalues of $\AA$ are in the strip
$${\cal B}=\lbrace \lb \in \C: -\al \leq \Re(\lb)<0 \rbrace.$$ 
Moreover the characteristic equation admits the following expansion 
\be  \label{f} f(\lb)=f_0(\lb)+\dfrac{f_1(\lb)}{\lb} +\dfrac{f_2(\lb)}{\lb^2}+\dfrac{f_3(\lb)}{\lb^3}+\go{4},\ee
where $f_i,i=0,...,3$ is a bounded function on ${\cal B}$ given by (\ref{fi}) below. 
\elem
\begin{proof} First, if $\lb$ is an eigenvalue of the operator $\AA$ associated to the normalized eigenvector $U$, from (\ref{e27}), $0 > \Re(\lb) = - \dfrac{k_2}{k_1} |\eta|^2 - \dfrac{k_4}{k_3} |\gamma|^2 \geq - k_2 - k_4$, since $\dfrac{1}{k_1} \cdot |\eta|^2$ and $\dfrac{1}{k_3} \cdot |\gamma|^2$ are both smaller than $\| U \|_{\mathcal{H}}^2 = 1$. Hence the existence of $\al$. \\
Furthermore $e^{t_i},i=1...,4$ is bounded as $|\lb| \longrightarrow \infty,$ where $t_i=t_i(\lb),i=1,...,4$ is given by (\ref{ti}). \\
\\
By Taylor series it holds

\beq \label{t1}  t_1=\lambda+ \frac{i \sqrt{b}}{2}+\frac{b}{8 \lambda }-\frac{i b^{\frac{3}{2}}}{16 \lambda ^2}+\go{3}, \eeq \\
\beq \label{t3} t_3=\lambda-\frac{i \sqrt{b}}{2}+\frac{b}{8 \lambda }+\frac{i b^{\frac{3}{2}}}{16 \lambda ^2}+\go{3}.
\eeq

Inserting (\ref{t1}) and (\ref{t3}) into (\ref{sys1}) and using Taylor series, after long calculations we get 
 
$$M(\lb)=\tilde{M}(\lb)+\go{3},$$ 

where $\tilde{M}(\lb)$ is a matrix which only contains terms of order $1,\dfrac{1}{\lb}$ or $\dfrac{1}{\lb^2}.$ Computing the determinant of $\tilde{M}(\lb)$ and keeping only the terms of order less than or equal to $\dfrac{1}{\lb^2},$ we get after lengthy  calculations

\be  \label{fbis} f(\lb)=f_0(\lb)+\dfrac{f_1(\lb)}{\lb} +\dfrac{f_2(\lb)}{\lb^2}+\dfrac{f_3(\lb)}{\lb^3}+\go{4},\ee

where $f_i,i=0,...,3$ is a bounded function given by

\beqs \label{fi}
f_0(\lb)&=& \dfrac{1}{4} \,e^{-t_1-t_3} (e^{t_1+t_3}-1)^2,\\
f1(\lb)&=& - \dfrac{1}{4} \left(2(k_2+k_4)-e^{t_1+t_3}(k_1+k_2+k_3+k_4)+e^{-t_1-t_3}(k_1+k_3-k_2-k_4) \right), \\
f_2(\lb)&=& - \dfrac{1}{16} \lbrace-4  (b+2 k_1 k_3-2 k_2 k_4)\\
&+&
(3 b-4 k_1 k_3-4 k_2 k_3-4 k_1 k_4-4 k_2 k_4)e^{t_1+t_3}\\&+&
(3 b-4 k_1 k_3+4 k_2 k_3+4 k_1 k_4-4 k_2 k_4)e^{-t_1-t_3}\\&+&
(-b+2 i \sqrt{b} k_1-2 i \sqrt{b} k_3)e^{t_1-t_3}\\&+&
(-b-2 i \sqrt{b} k_1+2 i \sqrt{b} k_3)e^{-t_1+t_3}\rbrace,\\
f_3(\lb)&=&- \dfrac{1}{16} \lbrace
-4 b (k_2+k_4)\\
&+&\frac{1}{2} b (7 k_1+6 k_2+3 k_3+6 k_4)e^{t_1+t_3}\\
&+&-\frac{1}{2} b (7 k_1-6 k_2+3 k_3-6 k_4)e^{-t_1-t_3}\\
&+&(-b k_2-2 i \sqrt{b} k_2 k_3-b k_4+2 i \sqrt{b} k_1 k_4)e^{t_1-t_3}\\
&+&(-b k_2+2 i \sqrt{b} k_2 k_3-b k_4-2 i \sqrt{b} k_1 k_4)e^{-t_1+t_3}\rbrace.\\
\eeqs
 \end{proof}
\blem \label{asympexp1} (Asymptotic behaviour of the large eigenvalues of $\AA$) \\
The large eigenvalues of $\AA$ can be split into two families $\left( \lb _k^j \right)_{k\in\Z , |k|\geq k_0}$, $j=1,2,$ ($k_0 \in \N,$ chosen large enough.) The following asymptotic expansions hold:

\be\label{dl0} \lb_k^1=i k\pi +o(1),\;\; \lb_k^2=i k\pi +o(1).\ee

Either $\lb_k^1=\lb_k^2$ and this root is of order 2, or $\lb_k^1\neq \lb_k^2$ and these two roots are simple. 
\elem

\begin{proof} 
The multiplicity of the roots of $f_0$ given by (\ref{fi}) is two and $\lb$ is a root of $f_0$ if and only if 

$$\exists k \in \Z, (t_1+t_3)(\lb)=2i k\pi.$$

Since $(t_1+t_3)(\lb)= 2\lb + \dfrac{b}{4\lb} + o \left( \dfrac{1}{\lb} \right),$ we deduce that, for each $k\in \Z,$ with $|k|$ large enough, corresponds a double root of $f_0,$ denoted by $\lb_k^0$ which satisfies   
$$\lb_k^0=i k\pi+\goko. $$
We will now use Rouch\'e's theorem. Let $B_k=B(i k\pi,r_k)$ be the ball of centrum $i k\pi$ and radius 
$r_k=\dfrac{1}{k^{1/4}}$ and $\lb \in \partial B_k$ (i.e $\lb= i k\pi +r_ke^{i\theta},\; \theta\in [0,2 \pi]$).
Then we successively have:

$$
\begin{array}{lll}
(t_1+t_3)(\lb)&=&2ik\pi+2 r_ke^{i\theta}+\goko, \\
e^{(t_1+t_3)(\lb)}&=&e^{2 r_ke^{i\theta}+\goko}\\
&=& 1 + 2 r_ke^{i\theta}+O(r_k^2),
\end{array}
$$
and 

$$
\begin{array}{lll}
f_0(\lb)&=&(1/4)(1 - 2 r_ke^{i\theta}+O(r_k^2))(2 r_ke^{i\theta}+ O(r_k^2))^2  \\
&=&(1/4)(1 - 2 r_ke^{i\theta}+O(r_k^2))(4 r_k^2 e^{2i\theta} + O(r_k^3))\\
&=&r_k^2 e^{2 i\theta}+O(r_k^3).
\end{array}
$$
It follows that there exists a positive constant $c$ such that 
$$\forall \lb\in \partial B_k,\; |f_0(\lb)|\geq c r_k^2 =\dfrac{c}{\sqrt{k}}.$$ 

Then we deduce from (\ref{f}) that $|f(\lb)-f_0(\lb)| = O(\dfrac{1}{\lb}) = \goko$. It follows that, for $|k|$ large enough 
$$\forall \lb \in B_k ,\; |f(\lb)-f_0(\lb)|<|f_0(\lb) |,$$
hence we get the result.
\end{proof}

\brk Since the imaginary axis is an asymptote  for the spectrum of $\AA,$ then System (\ref{sys1}) is not uniformly  stable.
\erk
\brk Obviously the previous asymptotic analysis of the spectrum is not necessary to deduce that System (\ref{sys1}) is not uniformly  stable. Indeed, using the compact perturbation result of Russell (see \cite{russell:75}), we directly see that the dissipative system (\ref{sys1}) is not uniformly stable.  
\erk

More information concerning the asymptotic behavior of the spectrum of $\AA$ is given by:

\bprop \label{asympexp2} (Asymptotic expansions for the eigenvalues of $\AA$ and $\AA_0$) \\
Assume Condition
$${\bf (C_1):}\;\;  k_1 \neq k_3 \mbox{ or }  \sqrt{b} \neq 2 k \pi, k \in \N^*.$$  

Then the large eigenvalues of the dissipative operator $\AA$ are simple and can be split into two families $\left( \lb _k^j \right)_{k\in\Z , |k|\geq k_0},j=1,2,$ ($k_0 \in \N,$ chosen large enough.) \\
\\
Moreover, we have the following asymptotic expansions for the eigenvalues of $\AA$:

\beq 
\label{dl1}\lb_k^1= i k \pi +i \dfrac{\al_1}{k} -\dfrac{\beta_1}{k^2}+\pok{2},  \\        
\label{dl2}\lb_k^2= i k \pi +i \dfrac{\al_2}{k} -\dfrac{\beta_2}{k^2}+\pok{2},
\eeq
where $\alpha_j\in \R,\; \beta_j>0,j=1,2.$ \\
\\
If Condition ${\bf (C_1)}$ above is still assumed, the large eigenvalues of the conservative operator $\AA_0$ are simple and can be split into two families $\left( \lb _k^{j,0} \right)_{k\in\Z , |k|\geq k^0_0},j=1,2,$ ($k^0_0 \in \N,$ chosen large enough) with the following asymptotic expansions:

\beq 
\label{dl01}\lb_k^{1,0}= i k \pi +i \dfrac{\al_1}{k} +\pok{2},  \\        
\label{dl02}\lb_k^{2,0}= i k \pi +i \dfrac{\al_2}{k}+\pok{2},
\eeq
with the same $\alpha_j$ as above.  
\eprop  

(cf. Figure $1$ of Section \ref{num}.)

\brk The explicit values for $\alpha_1$ and $\alpha_2$ are given by $(\ref{defalpha})$, $(\ref{defgamma1})$ and $(\ref{defgamma2})$. They only depend on the values of $b$, $k_1$ and $k_3$. As for $\beta_j$, it is defined by $\beta_j:=\dfrac{\om_2^j}{\om_1^j}$, $j=1;2$, with $\om_1^j$ and $\om_2^j$   given by $(\ref{om1})$ and $(\ref{om2})$. 
\erk

\begin{proof}

{\bf Step 1.} \\
Let $\lb_k=\lb_k^j,$ with $j=1$ or $j=2.$ From (\ref{dl1}), it follows $\lb_k = ik\pi+\eps_k,$ where $\eps_k=o(1).$ 

Using (\ref{t1}) and (\ref{t3}) leads to:

$$t_1+t_3=2 i k \pi+ 2 \eps_k -\dfrac{i b}{4 k \pi }+o(\eps_k) +\pok{2}+o \left( \dfrac{\eps_k}{k} \right),$$

which implies: 

\beq \label{z1} e^{t_1+t_3}=1-\dfrac{i b}{4 k \pi }-
\dfrac{b^2}{32 k^2 \pi ^2}-\dfrac{i b \eps_k}{2 k \pi }+2\eps_k+o(\eps_k) +\pok{2}+o \left( \dfrac{\eps_k}{k} \right),\\
\label{z2}
e^{-t_1-t_3}=1+\dfrac{i b}{4 k \pi }-
\dfrac{b^2}{32 k^2 \pi ^2}+\dfrac{i b \eps_k}{2 k \pi }-2\eps_k+o(\eps_k) +\pok{2}+o \left( \dfrac{\eps_k}{k} \right). 
\eeq

Similarly it holds 

$$t_1-t_3=i \sqrt{b}+\frac{i b^{3/2}}{8 k^2 \pi ^2}+\pok{2},$$

and we deduce that 

\beq \label{z3} e^{t_1-t_3}=e^{i \sqrt{b}}+\frac{i b^{3/2} e^{i \sqrt{b}}}{8 k^2 \pi ^2}+\pok{2},\\
\label{z4} e^{-t_1+t_3}=e^{-i \sqrt{b}}-\frac{i b^{3/2} e^{i \sqrt{b}}}{8 k^2 \pi ^2}+\pok{2}.\eeq

Using (\ref{f}), inserting (\ref{z1})-(\ref{z4}) into $f(\lb_k)$ and keeping only the terms greater than or equal to $O(\dfrac{1}{k^2})$, we obtain after  calculations 
% voir equacar19-11

\beq  \label{e30}
 f(\lb_k)= \eps_k^2 - i \gamma_1 \dfrac{\eps_k}{k} - \gamma_2 \dfrac{1}{k^2} + \pok{2} + o(\eps_k^2) + o(\dfrac{\eps_k}{k})=0,
\eeq 

where  

\beq  \label{defgamma1} \gamma_1=\dfrac{b+4 (k_1+k_3)}{4 \pi }, \eeq
\beq  \label{defgamma2} \gamma_2=\dfrac{-8 b+b^2+8 b k_1+8 b k_3+64 k_1 k_3+8 b \cos(\sqrt{b})+16 \sqrt{b} (k_1-k_3) \sin(\sqrt{b})}{64 \pi ^2}.\eeq

Multiplying (\ref{e30}) by $k^2$ leads to:

$$(k\eps_k)^2 -i \gamma_1 (k\eps_k) - \gamma_2 + o(1) + o(k \eps_k) + o(k^2 \eps_k^2) = 0.$$

Thus $k\eps_k$ is bounded and  

$$(k\eps_k)^2 -i \gamma_1 (k\eps_k) - \gamma_2 + o(1) = 0.$$

The previous equation has two solutions 

$$ k\eps_k =\dfrac{i}{2 }(\gamma_1-\sqrt{\gamma_1^2-4 \gamma_2})+o(1)\; \mbox {or } k\eps_k =\dfrac{i}{2 }(\gamma_1+\sqrt{\gamma_1^2-4 \gamma_2})+o(1).$$ 

Denoting by

\beq  \label{defalpha} \al_1=\dfrac{\gamma_1-\sqrt{\gamma_1^2-4 \gamma_2}}{2}\; \mbox{ and }\;\;
\al_2=\dfrac{\gamma_1+\sqrt{\gamma_1^2-4 \gamma_2}}{2},
\eeq

it holds:

$$ \eps_k = i\dfrac{\al_1}{k} +\poko\; \mbox{ or } \;\; \eps_k = i\dfrac{\al_2}{k} +\poko.$$

Note that, if Condition $({\bf C_1})$ holds, $\al_1$ and $\al_2$ are real numbers and $\al_1 \neq \al_2.$ 
Indeed $\gamma_1 \in \R$ and it holds

$$ \gamma_1^2-4 \gamma_2=\dfrac{b+2 (k_1-k_3)^2-b \cos(\sqrt{b}) -2 \sqrt{b} (k_1-k_3) \sin(\sqrt{b})}{2 \pi^2} $$
$$= \dfrac{1}{2 \pi^2} \left[2 \left(k_1 - k_3 - \dfrac{1}{2} \sqrt{b} \sin(\sqrt{b}) \right)^2 + \dfrac{b}{4} \left( (\cos(\sqrt{b}))^2 - 4 (\cos(\sqrt{b})) + 3) \right) \right] \geq 0$$

for all $k_1>0,k_3>0,b>0$. Thus $\gamma_1^2-4 \gamma_2=0,$ if and only if $k_1=k_3,$ and $\sqrt{b}=2 k \pi,k \in \N^*.$ \\
\\
Now it must be proved that near $i k \pi$, there are exactly two distinct roots, for $|k|$ great enough. For that purpose we consider $\Gamma_k$ the disk of center $z_k^0=i k \pi+i\dfrac{\al_1}{k}$ and radius $r_k=\dfrac{1}{2}\dfrac{|\al_1-\al_2|}{k}$  and the polynomial  $p_k$  defined by 

$$ p_k(z)= (z-ik\pi)^2-i\gamma_1 \dfrac{z-ik\pi}{k}-\gamma_2 \dfrac{1}{k^2}.$$

The roots of $p_k$ are $z_k^0$ and $i k \pi+i\dfrac{\al_2}{k}$ (it holds $\al_l^2 - \gamma_1 \al_l + \gamma_2 = 0$ since $\al_1 + \al_2 = \gamma_1$ and $\al_1 \al_2 = \gamma_2$). But $i k \pi+i\dfrac{\al_2}{k}$ does not belong to $\Gamma_k,$ if $|k|$ is large enough. 
Let $z=z_k^0+\dfrac{1}{2 k} (\al_2-\al_1)e^{i\theta}, \theta \in [0,2\pi]$ any element of   $\partial\Gamma_k.$ 
Then $p_k(z)$ is proved to be:

$$p_k(z)=-\dfrac{e^{i \theta} \left(e^{i \theta}-2\right) \left(\gamma_1^2-4 \gamma_2\right)}{4 k^2} ,$$

thus there exists a positive constant $c$ independent of $k$ such that 
$$|p_k(z)|\geq \dfrac{c}{k^2},\; \forall z\in \partial \Gamma_k.$$

On the other hand, using $(\ref{e30})$ we get $|f(z)-p_k(z)|=\pok{2}.$ Therefore, Rouch\'e's theorem implies that $f$ has only one root in $\Gamma_k,$ if $k$ is large enough. Finally, we have proved that the large eigenvalues of $\AA$ are simple and can be split into two families with the following expansions: 
$$\lb_k^1=ik\pi+ i \dfrac{\al_1}{k}+\poko,\; \lb_k^2=ik\pi+ i\dfrac{\al_2}{k}+\poko. $$

Note that the eigenvalues of the conservative operator $\AA_0$ have the same asymptotic expansions, since $\al_1$ and $\al_2$ are independent of the values of $k_2$ and $k_4$. \\  
\\

{\bf Step 2.}\\ Since for $j=1,2,\al_j \in \R,$ we need one more term in the expansion of $\lb_k^j,j=1,2.$ 

From Step 1, the expansion for $j=1$ or $j=2$ is: 

$$\lb_k^j = ik \pi +i \dfrac{\al_j}{k}+\dfrac{\eps_k^j}{k},$$ where $\eps_{k}^j=o(1).$ 

Using (\ref{f}), Taylor series and simplification in the term of order $\dfrac{1}{k^2}$ coming from Step 1, we get after a long calculation

\be \label{dl3} f(ik \pi+ i \dfrac{\al_j}{k}+\dfrac{\eps_k^j}{k})=\dfrac{1}{k^2} (\om_1^j \eps_{k}^j + (\eps_{k}^j)^2)+\om_2^j \dfrac{1}{k^3}+\pok{3}=0
\ee

%$$f(ik\pi+\dfrac{\al_1}{k}+\dfrac{\eps_k}{k})=i \dfrac{\om_1^i e_k}{k^2}+i \dfrac{\om_2^i}{k^3}+\pok{3}+\dfrac{1}{k^2}o(\eps_k)=0,$$
where  $\om_l^j \in i\R, l=1,2$ and is given by 
%cf 15juin2013%
\beq \label{om1}  \om_1^j = - \dfrac{i(b+4 k_1+4 k_3-8  \al_j \pi )}{4\pi }=\mp i \sqrt{\gamma_1^2-4  \gamma_2}, \eeq
\beq \label{om2} \om_2^j=\dfrac{i}{8 \pi^3}  ( \gamma_3-8\pi  (k_1k_2+k_3k_4) \al_j),
\eeq 
$j=1,2,$ where 

\be \label{defgamma3} \gamma_3=b( k_1 k_2+k_3k_4)+8 k_1 k_3(k_2+k_4)+2 \sqrt{b} (k_1k_2-k_3k_4)\sin(\sqrt{b}). \ee 

Since we assume $({\bf C_1})$ then $\om_1^j\neq 0$ (see the remark just below (\ref{defgamma2})) and we deduce from (\ref{dl2}) that $\eps_k^j=-\dfrac{\om_2^j}{\om_1^j}\dfrac{1}{k} +o \left( \dfrac{1}{k} \right).$ Setting  
 $\beta_j=\dfrac{\om_2^j}{\om_1^j},$ then it holds $\beta_j \in \R$ and (\ref{dl1}) holds. Since all the eigenvalues of $\AA$ are on the left of the imaginary axis, necessarily $\beta_j\geq 0.$ \\
 \\
Note that, if $k_2=k_4=0$ (associated conservative operator $\mathcal{A}_0$), $\gamma_3=0$ and thus, $\om_2^j$ and $\beta_j$ vanish as well.  \\ 
 \\ 
Now, if $(k_2,k_4) \neq (0,0)$ (dissipative operator $\mathcal{A}$), $\beta_j\neq 0, j=1,2$ as it is proved below. \\
 \\

{\bf Step 3.}\\ Assume that $(k_2,k_4) \neq (0,0)$ and $\beta_j=  0, j=1,2.$ Then $\om_2^j=0, j=1,2,$ thus 
$$\alpha_j =\dfrac{\gamma_3 }{8\pi  (k_1k_2+k_3k_4)}, j=1,2.$$
But, since $\alpha_1+ \alpha_2 = \gamma_1$ and $\alpha_1 \cdot
\alpha_2 = \gamma_2$, it holds:
$$\alpha_j ^2 -\gamma_1 \alpha_j+\gamma_2=0, j=1,2.$$

It follows 

$$\left( \dfrac{\gamma_3 }{8\pi  (k_1k_2+k_3k_4)} \right)^2-\gamma_1 \dfrac{\gamma_3 }{8\pi  (k_1k_2+k_3k_4)}+\gamma_2=0.$$

We multiply the previous identity by $16 (k_1 k_2 + k_3 k_4)^2 \pi^2$ and use $(\ref{defgamma2})$ and $(\ref{defgamma3})$ to get: 

$$\gamma_4+\gamma_5 b \cos(\sqrt{b})+\gamma_6 b \sin^2(\sqrt{b})+\gamma_7 \sqrt{b} \sin(\sqrt{b})=0,$$
where $\gamma_4=-2 b (k_1 k_2 + k_3 k_4)^2 - 16 k_1 k_2 k_3 k_4 (k_1 - k_3)^2,\gamma_5=2  (k_1 k_2 + k_3 k_4)^2,$

$\gamma_6 =(k_1 k_2 - k_3 k_4)^2$ and $\gamma_7 =16 k_1 k_2 (k_1 - k_3) k_3 k_4.$ \\
\\
Now, using the fact that $\gamma_1^2-4\gamma_2>0$  or equivalently  $2\pi^2(\gamma_1^2-4\gamma_2)>0$ (which is true if and only if Condition ${\bf (C_1)}$), it holds
 
$$ b + 2 k_1^2 - 4 k_1 k_3 + 2 k_3^2 - b \cos(\sqrt{b}) - 2 \sqrt{b} (k_1 - k_3)  \sin(\sqrt{b}) >0.$$

Thus, using the definition of $\gamma_7$,  

$$\gamma_7 \sqrt{b} \sin(\sqrt{b})<8 k_1k_2k_3k_4( b + 2 (k_1 - k_3)^2 - b \cos(\sqrt{b})).$$

We get after simplifications 

$$\begin{array}{lll}
0&=&\gamma_4+\gamma_5 b \cos(\sqrt{b})+\gamma_6 b \sin^2(\sqrt{b})+\gamma_7 \sqrt{b} \sin(\sqrt{b})\\
&<& (2(k_1 k_2 + k_3 k_4)^2 - 8 k_1 k_2 k_3 k_4) (-b + b \cos(\sqrt{b}) + (k_1 k_2 - k_3 k_4)^2 b \sin^2(\sqrt{b}))\\
&<& b (k_1 k_2 - k_3 k_4)^2 (-2 + 2 \cos(\sqrt{b}) + \sin^2(\sqrt{b}))\\
&=& -4 b (k_1 k_2 - k_3 k_4)^2 \left(\sin(\dfrac{\sqrt{b}}{2}) \right)^4.
\end{array}$$

Since this inequality never holds, the assumption $\beta_j=0, j=1;2$ does not hold either. 
%\hfill \fin  \\ \\
% cf 11fev.nb

\end{proof}

Now, if Condition ${\bf (C_1)}$ does not hold, the calculations are different (and long). The details are not given here. The results are given without proofs.

\bprop \label{asympexp3} (Asymptotic expansions for the eigenvalues of $\AA$ and $\AA_0$ - particular cases) \\

\begin{enumerate}
\item \bf Case $k_1=k_3,\;\; k_2 \neq k_4  \;\; b=4 p^2 \pi^2, \;\;p\in \N^*.$ \it \\
The large eigenvalues of the dissipative operator $\AA$ are simple and can be split into two families $\left( \lb _k^j \right)_{k\in\Z , |k|\geq k_0},j=1,2,$ ($k_0 \in \N,$ chosen large enough.) Moreover they satisfy the following asymptotic expansions: 

\be\label{cp1} \lb_k^1=i k \pi+\dfrac{i (2 k_1+p^2 \pi ^2)}{2 k \pi }-\dfrac{k_1 k_2}{k^2 \pi ^2}+ o \left( \dfrac{1}{k^2} \right),
\ee
\be\label{cp2} \lb_k^2=i k \pi+\dfrac{i (2 k_1+p^2 \pi ^2)}{2 k \pi }-\dfrac{k_1 k_4}{k^2 \pi ^2}+ o \left( \dfrac{1}{k^2} \right).
\ee

(cf. the table and Figure $2$ of Section \ref{num}.)

\item \bf Case $k_1=k_3,\;\; k_2=k_4 \neq 0,  \;\; b=4 p^2 \pi^2, \;\;p\in \N^*.$ \it \\
The large eigenvalues of the dissipative operator $\AA$ can be split into two families $\left( \lb _k^j \right)_{k\in\Z , |k|\geq k_0}$, $j=1,2,$ ($k_0 \in \N,$ chosen large enough.) Moreover they satisfy the following asymptotic expansions: 

$$\lb_{k}^j=i k \pi +\frac{i \left(2 k_1+p^2 \pi ^2\right)}{2 k \pi }-\dfrac{k_1k_2}{k^2 \pi ^2}+\dfrac{ i (a_{3,j}-24k_1k_2^2)}{24 k^3 \pi ^3}
+o \left( \dfrac{1}{k^3} \right),$$
where $a_{3,j} \in \R,j=1,2$ are given below.

\item \bf Case $k_1=k_3,\;\; k_2= k_4=0 , \;\; b=4 p^2 \pi^2, \;\;p\in \N^*.$ \it \\
The large eigenvalues of the conservative operator $\AA_0$ can be split into two families $\left( \lb _k^{j,0} \right)_{k\in\Z , |k|\geq k^0_0},j=1,2,$ ($k^0_0 \in \N,$ chosen large enough) with the following asymptotic expansions:

$$\lb_{k}^{j,0} = i k \pi +\frac{i \left(2 k_1+p^2 \pi ^2\right)}{2 k \pi }+\frac{i a_{3,j}}{24 k^3 \pi ^3}
+o \left( \dfrac{1}{k^3} \right),$$
where $a_{3,j} \in \R,j=1,2$ are given below

$$a_{3,1} =-24 k_1^2-8 k_1^3-36 k_1 p^2 \pi ^2+9 p^4 \pi ^4-12 p \pi  \sqrt{4 k_1^4-43 k_1^2 p^2 \pi ^2-4 k_1 p^4 \pi ^4+p^6 \pi ^6},$$

$$a_{3,2} =-24 k_1^2-8 k_1^3-36 k_1 p^2 \pi ^2+9 p^4 \pi ^4+12 p \pi  \sqrt{4 k_1^4-43 k_1^2 p^2 \pi ^2-4 k_1 p^4 \pi ^4+p^6 \pi ^6}.$$

\rm Note that, if $4 k_1^4-43 k_1^2 p^2 \pi ^2-4 k_1 p^4 \pi ^4+p^6 \pi ^6 \neq 0$ then $\lb_{k}^{1,0} \neq \lb_{k}^{2,0}$ for $k$ large enough. Idem for $\lb_{k}^{1}$ and $\lb_{k}^{2}$ of the previous case.
\end{enumerate}
\end{proposition}

%%%%%%%%%%%%%%%%%%%%%%%%%%%%%%%%%%%%%%%%%%%%%%%%%%%%%%%%%%%%%%%%%%%%%%%%%%%%%%%%%%%%%%%%%%%%%%%%%%%%%%%%%%%%%%%%%%%%%%%%%%%%%%%%%%%%%%%%%%%%%%   %%%%%%%      Riesz basis
%%%%%%%%%%%%%%%%%%%%%%%%%%%%%%%%%%%%%%%%%%%%%%%%%%%%%%%%%%%%%%%%%%%%%%%%%%%%%%%%%%%%%%%%%%%%%%%%%%%%%%%%%%%%%%%%%%%%%%%%%%%%%%%%%%%%

\section{Riesz basis} 

\label{secriesz}

\noindent In this section, it is proved that the system ${\cal F}$ of generalized eigenvectors of the dissipative operator $\mathcal{A}$ (introduced in Theorem \ref{ppr}) forms a Riesz basis of $\mathcal{H}$. To this end, we use Theorem 1.2.10 of \cite{abd} which is a rewriting of Guo's version of Bari Theorem with another proof (see \cite{Guo}). \\
For the sake of completeness, Theorem 1.2.10 of \cite{abd} is recalled~:

\begin{theorem} Let $\mathcal{A}$ be a densely defined operator in a Hilbert space $\mathcal{H}$ with compact resolvent. Let $\{ \phi_n \}_{n=1}^{\infty}$ be a Riesz basis of $\mathcal{H}$. If there are two integers $N_1$, $N_2 \geq 0$ and a sequence of generalized eigenvectors $\{ \psi_n \}_{n = N_1 + 1}^{\infty}$ of $\mathcal{A}$ such that 

$$\sum_{n=1}^{\infty} \| \phi_{n+N_2} - \psi_{n+N_1} \|_2^2 < \infty,$$

then the set of generalized eigenvectors (or root vectors) of $\mathcal{A}$, $\{ \psi_n \}_{n = 1}^{\infty}$ forms a Riesz basis of $\mathcal{H}$. 
\end{theorem}

The family ${\cal F}_0 $ of eigenvectors of the conservative operator $\mathcal{A}_0$ is an orthornormal basis of the Hilbert space ${\caH}.$ Thus, it is enough to show that the eigenfunctions of $\mathcal{A}_0$ associated to the eigenvalues $\lb_{k}^{j,0}\in \sigma_0^2$ and those of the dissipative operator $\mathcal{A}$ associated to the eigenvalues $\lb_{k}^{j}\in\sigma^2$ are quadratically close to one another. 

%%%%%%%%%%%%%%%%%%%%%%%%%%%%%%%%%%%%%%%%%%%%%%%%%%%%%%%%%%%%%%%%%%%
%%%% Theorem Riesz basis
%%%%%%%%%%%%%%%%%%%%%%%%%%%%%%%%%%%%%%%%%%%%%%%%%%%%%%%%%%%%%%%%%%%

\begin{theorem} \label{rieszbasis} (Riesz basis for the operator $\mathcal{A}$) \\ 
For any $j \in \{ 1 ; 2 \}$, it holds:
$$\sum_{|k|\geq k_0} \| \phi_k^j - \psi_k^j \|_{\mathcal{H}}^2 < \infty.$$
Thus, the set ${\cal F}$ of generalized eigenvectors of $\mathcal{A}$ forms a Riesz basis of $\mathcal{H}$. 
\end{theorem}

%%%%%%%%%%%%%%%%%%%%%%%%%%%%%%%%%%%%%%%%%%%%%%%%%%%%%%%%%%%%%%%%%%%
%%%% Theorem Riesz basis
%%%%%%%%%%%%%%%%%%%%%%%%%%%%%%%%%%%%%%%%%%%%%%%%%%%%%%%%%%%%%%%%%%%

\begin{proof}

{\bf Step 1.} \\
Since $\psi_k^j$ lies in $\mathcal{H}$, it has six components (see Section \ref{wellposedness}). Let us write $\psi_k^j:=(u_k^j,v_k^j,y_k^j,z_k^j,\eta_k^j,\gamma_k^j)$ and let us first prove that 

\be \label{etakj} |\eta_k^j| = \goko, \, |\gamma_k^j| = \goko. \ee

From (\ref{e27}), it follows 

\begin{equation}
\Re <\AA \psi_k^j, \psi_k^j>_{\HH} = -\dfrac{k_2}{k_1} \vert \eta_k^j\vert^2 - \dfrac{k_4}{k_3} \vert \gamma_k^j \vert^2.  
\end{equation}

Now, $\Re <\AA \psi_k^j, \psi_k^j>_{\HH}$ is also equal to $\Re (\lb_{k}^{j}) = - \dfrac{\beta_j}{k^2} + \pok2 = \gok2$ due to (\ref{dl1}) and (\ref{dl2}). Hence (\ref{etakj}). \\
\\
{\bf Step 2. Projection.} 
Let $j=1,2$ and $k,\;|k|\geq k_0$ be fixed and denote by $P_k^j$ the orthogonal projection on  $\{\phi_k^j\}^\bot,$  the orthogonal space of the  1-dimensional space directed by $\phi_k^j.$
Clearly there exists $\alpha_k^j$ which can be supposed to satisfy $\alpha_k^j\geq 0$ without loss of generality, such that   
\be \label{proj} \psi_k^j= \alpha_k^j \phi_k^j + R_k^j, \ee
where $R_k^j=P_k^j(\psi_k^j)$ and $R_k^j$ is orthogonal to $\phi_k^j$.   \\
Thus, due to Lemma \ref{techlem} given later, $$1 = \| \psi_k^j \|_{\mathcal{H}}^2 = |\alpha_k^j|^2 \cdot \| \phi_k^j \|_{\mathcal{H}}^2 + \| R_k^j \|_{\mathcal{H}}^2 = |\alpha_k^j|^2 + \gok2.$$ Then, $\exists c_k^j$, real number bounded with respect to $k$, such that $$\alpha_k^j = \sqrt{1- \dfrac{c_k^j}{k^2}} = 1 -  \dfrac{c_k^j}{2 k^2} + \pok2 = 1 + \gok2.$$ 

{\bf Step 3: $\{\phi_k^j\}_{|k|\geq k_0}$ and $\{\psi_k^j\}_{|k|\geq k_0}$ are quadratically close to one another.} \\
Using Step 2, $$\| \phi_k^j - \psi_k^j \|_{\mathcal{H}}^2 = \|(\alpha_k^j - 1) \phi_k^j + R_k^j \|_{\mathcal{H}}^2 = |\alpha_k^j - 1|^2 \cdot \|\phi_k^j \|_{\mathcal{H}}^2 + \| R_k^j \|_{\mathcal{H}}^2 = \gok4 + \gok2 = \gok2.$$
Hence 
$$\sum_{|k|> 0} \| \phi_k^j - \psi_k^j \|_{\mathcal{H}}^2 < \infty. $$
\end{proof}

\blem (Technical Lemma for the proof of Theorem \ref{rieszbasis}) \label{techlem} \\
The vector $R_k^j$, defined in Step 2 of the proof of Theorem \ref{rieszbasis}, satisfies: $\| R_k^j \|_{\mathcal{H}} = \goko$, for $j=1,2$. 
\elem

\begin{proof}
Using (\ref{proj}), it holds, for $j=1;2$:

$$\begin{array}{ll}
(\mathcal{A}_0 - \lb_k^j) \psi_k^j & = (\mathcal{A}_0 - \lb_k^j) (\alpha_k^j \phi_k^j + R_k^j) = \alpha_k^j (\mathcal{A}_0 - \lb_k^j)(\phi_k^j) + (\mathcal{A}_0 - \lb_k^j) (R_k^j) \\
& = k_2 \eta_k^j e_5 + k_4 \gamma_k^j e_6 \; \mbox{(this follows from} \; (\ref{AandA0})).
\end{array}$$

Since $\mathcal{A}_0$ and $P_k^j$ commute, then applying $P_k^j$ to the previous identity, we get 
$$(\mathcal{A}_0 - \lb_k^j)(R_k^j) = k_2 \eta_k^j P_k^j(e_5) + k_4 \gamma_k^j  P_k^j(e_6).$$
Thus $$R_k^j = k_2 \eta_k^j (\mathcal{A}_0 - \lb_k^j)^{-1} P_k^j(e_5) + k_4 \gamma_k^j (\mathcal{A}_0 - \lb_k^j)^{-1} P_k^j(e_6).$$

Writing $e_5$ in the orthonormal basis ${\cal F}_0,$ it follows 

$$\begin{array}{lll}
P_k^j(e_5)&=&\displaystyle\sum_{i\in I_0,l=0,...,m_i-1} <\tilde{\phi_i^l},e_5>_{\mathcal{H}}\tilde{\phi_i^l} \\
&+&\displaystyle \sum_{|l|\geq k_0, l \neq k} \left[<e_5, \phi_l^j>_{\mathcal{H}} \cdot \phi_l^j +<e_5, \phi_l^{j+1}>_{\mathcal{H}} \cdot \phi_l^{j+1} > \right]\\
&+& <e_5, \phi_k^{j+1}>_{\mathcal{H}} \cdot \phi_k^{j+1},
\end{array}$$
where the exponent $j$ is defined modulo $2$.\\
Then 
$$
\begin{array}{lllll}
\| (\mathcal{A}_0 - \lb_k^j)^{-1} (P_k^j(e_5)) \|_{\mathcal{H}} & \leq& \displaystyle \sum_{i\in I_0,l=0,...,m_i-1} |<e_5,\tilde{\phi_i^l}>_{\mathcal{H}}|\cdot \dfrac{1}{|\kappa_i^{0} - \lb_k^j|}   \\
&+&\displaystyle \sum_{|l|\geq k_0, l \neq k} \left[|<e_5, \phi_l^j>_{\mathcal{H}}| \cdot \dfrac{1}{|\lb_l^{j,0} - \lb_k^j|} \right. \\
&+&\left. |<e_5, \phi_l^{j+1}>_{\mathcal{H}}| \cdot \dfrac{1}{|\lb_l^{j+1,0} - \lb_k^j|}  > \right]\\
&+& |<e_5, \phi_k^{j+1}>_{\mathcal{H}}| \cdot \dfrac{1}{|\lb_k^{j+1,0} - \lb_k^j|}   \\
&\leq& C \| e_5 \|_{\mathcal{H}} + |<e_5, \phi_k^{j+1}>_{\mathcal{H}}| \cdot \dfrac{1}{|\lb_k^{j+1,0} - \lb_k^j|},
\end{array}
$$
and similarly 
$$
\begin{array}{lll}
\| (\mathcal{A}_0 - \lb_k^j)^{-1} (P_k^j(e_6)) \|_{\mathcal{H}} & \leq&  C \| e_6 \|_{\mathcal{H}} + |<e_6, \phi_k^{j+1}>_{\mathcal{H}}| \cdot \dfrac{1}{|\lb_k^{j+1,0} - \lb_k^j|}.
\end{array}
$$

The existence of the constant $C$ independent of $k$ in the latest expressions comes from the fact that $l$ is different from $k$ in the sum. Indeed the behaviour of $|\lb_l^{j,0} - \lb_k^j|$ and that of $|\lb_l^{j+1,0} - \lb_k^j|$ are given by (\ref{dl1}), (\ref{dl2}), (\ref{dl01}) and (\ref{dl02}). They are both bounded from below by a constant independent of $k$ and $l$. Note that this still holds in the particular cases described by Proposition \ref{asympexp3}.\\
\\
The expression $|\lb_k^{j+1,0} - \lb_k^j|$ is not bounded from below by a constant independent of $k$. The same asymptotic expansions prove that, for $j=1;2$:

\be |\lb_k^{j+1,0} - \lb_k^j| = \goko. \ee

Thus, using (\ref{etakj}), the result follows as soon as it has been proved, for $j=1;2$: 

\be \label{lastesti} 
|<e_5, \phi_k^{j+1}>_{\mathcal{H}}|= \goko \; \mbox{and} \; |<e_6, \phi_k^{j+1}>_{\mathcal{H}}| = \goko.
\ee

Since $\phi_k^j:= (u_k^{j,0},v_k^{j,0},y_k^{j,0},z_k^{j,0},\eta_k^{j,0},\gamma_k^{j,0}) \in D(\caA),$ and $e_5=(0,0,0,0,1,0)$ then 
\be\label{cinq} <e_5, \phi_k^{j+1}>_{\mathcal{H}}= \eta_k^{j,0}=v_k^{j,0}(1),\ee 
and $(u_k^{j,0},v_k^{j,0},y_k^{j,0},z_k^{j,0})$ is solution of System (\ref{syseig}) with $\lb = \lb_k^{j,0} = i h_k^{j,0}, h_k^{j,0} \in \R$. In particular $(i)$ is

$$(u_k^{j,0})_{xx}+(y_k^{j,0})_x = -(h_k^{j,0})^2 u_k^{j,0}.$$

For simplicity, the indices and exponents are dropped from now on. 

$$\begin{array}{ll}
\dint_0^1 (u_{xx} + y_x)(x) \cdot (\overline{u_x+y})(x) dx &=- \dint_0^1 h^2 u(x) \cdot (\overline{u_x+y})(x) dx \\
&= - h^2 \left( \dint_0^1 u(x)\overline{u_x}(x) dx + \dint_0^1 u(x)\overline{y}(x) dx \right).
\end{array}$$

Integrating by parts, it follows $$\dint_0^1 u(x) \overline{u_x}(x) dx = - \dint_0^1 u_x(x) \overline{u}(x) dx + |u(1)|^2 + |u(0)|^2$$ and, due to $(iii)$ of System (\ref{syseig}):
$$2 \Re \left \{ \int_0^1 u(x) \cdot \overline{u_x}(x) dx \right \} = |u(1)|^2. $$ 
And thus 
\be \label{first} 2 \Re \left \{ \int_0^1 (u_{xx} + y_x)(x) (\overline{u_x+y})(x) dx \right \} = - h^2 \left( |u(1)|^2 + 2 \Re \left \{ \int_0^1 u(x) \overline{y}(x) dx \right \} \right). \ee
On the other hand, after an integration by parts, it holds:
$$\begin{array}{lll}
\dint_0^1 (u_{xx} + y_x) \cdot (\overline{u_x+y}) dx&= & - \dint_0^1 (u_{x} + y) \cdot (\overline{u_{xx}+y_x}) dx  \\
&+&|u_x(1) + y(1)|^2 - |u_x(0) + y(0)|^2
\end{array}
$$
which implies
\be \label{second} 2 \Re \left \{ \int_0^1 (u_{xx} + y_x)(x) \cdot (\overline{u_x+y})(x) dx \right \} = |u_x(1) + y(1)|^2 - |u_x(0) + y(0)|^2. \ee 

Now, using $v = \lb u$ and $z= \lb y$ (cf. the system just before System (\ref{syseig})), (\ref{first}) and (\ref{second}) imply

\be \label{third} |u_x(1) + y(1)|^2 + h^2 |u(1)|^2 = |u_x(0) + y(0)|^2 - 2 \Re \left \{ \int_0^1 v(x) \overline{z}(x) dx \right \}. \ee  

Then, $(v)$ of System (\ref{syseig}) with $k_2=0$ ($\phi_k^j$ is an eigenfunction of $\mathcal{A}_0$) leads to \\
$k_1^2 \cdot |u_x(1) + y(1)|^2 = |\lb|^4 \cdot |u(1)|^2 = h^4 \cdot |u(1)|^2 = h^2 \cdot |v(1)|^2$. \\
And \be\label{quatre}|u_x(1) + y(1)|^2 + h^2 |u(1)|^2 = (1 +\dfrac{h^2}{k_1^2}) |v(1)|^2.\ee
Using $(iv)$ of System (\ref{syseig}) as well as the trace Theorem applied to $u_x$ implies that there exists a constant $C_1$ such that: \\
\be\label{six}|u_x(0) + y(0)|^2 = |u_x(0)|^2 \leq C_1 \left( \| u_{xx} \|_{L^2(\Omega)}^2 + \| u_{x} \|_{L^2(\Omega)}^2 \right).\ee
Now $(i)$ of system (\ref{syseig}) gives:
$$\begin{array}{ll}
\| u_{xx} \|_{L^2(\Omega)}^2 & = \| \lb^2 u - y_x \|_{L^2(\Omega)}^2 \leq |\lb|^2 \cdot \| u \|_{L^2(\Omega)}^2 + \| y_x \|_{L^2(\Omega)}^2 \\ 
& \leq \| v \|_{L^2(\Omega)}^2 + b \| \phi \|_{\mathcal{H}}^2 \leq \| \phi \|_{\mathcal{H}}^2 + b \cdot \| \phi \|_{\mathcal{H}}^2 \leq 1 + b.
\end{array}$$
And
$$\begin{array}{ll}
\| u_{x} \|_{L^2(\Omega)}^2 & \leq \| u_x + y \|_{L^2(\Omega)}^2  + \| y \|_{L^2(\Omega)}^2 \\
& \leq \| \phi \|_{\mathcal{H}}^2 +  \dfrac{1}{|\lb|^2} \| z \|_{L^2(\Omega)}^2 \leq 1 + \dfrac{b}{|\lb|^2} \leq 1 + \dfrac{b}{h^2}.
\end{array}$$

Using successively the two previous estimates in (\ref{six}),  the Cauchy-Schwarz inequality applied to the last term of the right-hand side of (\ref{third}), (\ref{quatre}) and (\ref{cinq}), we get the first result of (\ref{lastesti}):
\be\label{concl} 
 |<e_5,\phi_k^j>_{\mathcal{H}}|\leq \left(\dfrac{C_1(2+\frac{b}{h^2})+2}{1+\frac{h^2}{k_1^2}}\right)^{1/2}. \ee  
Indeed, by definition, $h_k^{j,0}$ is the imaginary part of $\lb_k^{j,0}$ which behaves like $k$ for large values of $k$ (cf. Propositions \ref{asympexp2} and \ref{asympexp3}). \\
\\
To end this proof, let us give the sketch of the proof of the second estimate of (\ref{lastesti}). The ideas are similar to those developed just before. That is why the details are not given here. \\
It holds $<e_6,\phi_k^{j+1}>_{\mathcal{H}} = \gamma_k^{j,0} = z_k^{j,0}(1)$ with the same notation as before. \\
Integrations by parts allow to write the analogous of (\ref{third}):
\be \label{thirdbis} |y_x(1) - b u(1)|^2 + (h^2 - b) |y(1)|^2 = |y_x(0) - b u(0)|^2 + 2 b (h^2 - b) \cdot \Re \left \{ \int_0^1 y \overline{u} dx \right \}. \ee  
Long calculations, using System (\ref{syseig}), the Cauchy-Schwarz inequality as well as the trace Theorem applied to $y_x$, lead to the existence of a constant $C_2$ such that:

$$\left( \dfrac{h^2}{k_3^2} + 1 - \dfrac{b}{h^2} \right) |z(1)|^2 \leq b((b + 2) C_2 + 2 b) + b^2 (C_2 (b+1) + 2b) \dfrac{1}{h^2} + \dfrac{b^2}{h^2} |v(1)|^2.$$

Using (\ref{concl}), it follows: $|<e_6, \phi_k^{j}>_{\mathcal{H}}| = \goko.  $ 
\end{proof}

%%%%%%%%%%%%%%%%%%%%%%%%%%%%%%%%%%%%%%%%%%%%%%%%%%%%%%%%%%%%%%%%%%%%%%%%%%%%%%%%%%%%%%%%%%%%%%%%%%%%%%%%%%%%%%%%%%%%%%%%%%%%%%%%%%%%%%%%%%%%%%   Polynomial decay rate
%%%%%%%%%%%%%%%%%%%%%%%%%%%%%%%%%%%%%%%%%%%%%%%%%%%%%%%%%%%%%%%%%%%%%%%%%%%%%%%%%%%%%%%%%%%%%%%%%%%%%%%%%%%%%%%%%%%%%%%%%%%%%%%%%%%%

\section{Polynomial decay rate of the energy} 

\label{secdecay}

The energy is already known to be not uniformly stable (cf. Lemmas \ref{asympexp0} and \ref{asympexp1} and the remarks just below the lemmas). It is now proved to decay polynomially. To this end, the solution is explicitly expressed using the Riesz basis ${\cal F}$ of generalized eigenvectors of $\mathcal{A}$ (cf. Theorem \ref{rieszbasis}).

\begin{theorem} \label{poldecay} (Polynomial decay rate of the energy) \\ 
Assume that $a=1$ in System (\ref{e11})-(\ref{e16}). Then there exists a constant $C > 0$ such that for any initial datum $U_0 \in D(\mathcal{A})$, the energy of the system rewritten as (\ref{P}) satisfies the following estimate:

$$E(t) \leq C \cdot \dfrac{\| U_0 \|^2_{D(\mathcal{A})}}{t}, \forall t > 0,$$

where $E(t)=\dfrac{1}{2}\|U(t)\|^2_\HH.$ 
\end{theorem}

\begin{proof} Using the Riesz basis ${\cal F}$ (cf. Theorem \ref{rieszbasis}), we can write 
$$U_0=\sum_{i\in I} \sum_{k=1}^{n_i} \sum_{l=1} ^{\delta_{ik}} (u_0)_{ik}^l  \tilde{\psi_{ik}^l}
+ \sum_{|l|\geq k_0, j=1;2} (u_0)_l^j \psi_l^j.$$
The solution of (\ref{P}) is:
$$U(t)= \sum_{i\in I} e^{\kappa_i t}\left[ \sum_{k=1}^{n_i} \sum_{l=1} ^{\delta_{ik}} 
\left(\sum_{p=l}^{\delta_{ik}} (u_0)_{ik}^p \cdot \dfrac{t^{p-l}}{(p-l)!}\right) \; \tilde{\psi_{ik}^l}\right]
+ \sum_{|l|\geq k_0, j=1;2} e^{\lambda_l^j t} \cdot (u_0)_l^j\; \psi_l^j.
$$
Since ${\cal F}$ is a Riesz basis, there exists a positive constant $K$ such that the energy satisfies, for any $t>0$: 
$$E(t) \leq K \left[ \sum_{i\in I} e^{2 \Re(\kappa_i) t}\cdot \max \{ t^{m_i-1};1\}\cdot \left( \sum_{k=1}^{n_i} \sum_{l=1} ^{\delta_{ik}} 
 |(u_0)_{ik}^l|^2 \right)
+ \sum_{|l|\geq k_0, j=1;2} e^{2 \Re(\lambda_l^j) t}\cdot |(u_0)_l^j|^2\right].
$$
Using the asymptotic analysis performed in Propositions \ref{asympexp2} and \ref{asympexp3}  and since  $\Re(\lambda_l^j)<0,$ for all $|l| \geq k_0, j=1,2,$ it follows that   
$$\begin{array}{lll}
E(t) &\leq&  K  \max_{i\in I} \{ t^{m_i-1};1\} \cdot e^{-2 \min_{i \in I} (\Re(|\kappa_i|) t}
 \displaystyle \sum_{i\in I}\sum_{k=1}^{n_i} \sum_{l=1} ^{\delta_{ik}} 
 |(u_0)_{ik}^l|^2 \\
&+& K\displaystyle \sum_{|l|\geq k_0, j=1;2} \dfrac{e^{- \tilde{\beta_j} t/ l^2}}{l^2}\cdot l^2\cdot |(u_0)_l^j|^2,
\end{array} $$
where $ \tilde{\beta_j},j=1,2$ are positive constants. \\
Now, if $\beta >0$ is fixed, the function $u \mapsto u \cdot e^{-  \beta u}$ is a bounded function on $\R^+$. And

$$ \sum_{i\in I} \sum_{k=1}^{n_i} \sum_{l=1} ^{\delta_{ik}} 
 |(u_0)_{ik}^l|^2 
 \lesssim \| U_0 \|^2_{\mathcal{H}}, \; \mbox{and} \; \sum_{|l| \geq k_0, j=1;2} l^2 \cdot |(u_0)_l^j|^2  \lesssim  \| U_0 \|^2_{D(\mathcal{A})}.$$

Hence the result.
\end{proof}

%%%%%%%%%%%%%%%%%%%%%%%%%%%%%%%%%%%%%%%%%%%%%%%%%%%%%%%%%%%%
%%%%%    Numerical validation
%%%%%%%%%%%%%%%%%%%%%%%%%%%%%%%%%%%%%%%%%%%%%%%%%%%%%%%%%%%%

\section{Numerical validation} \label{num}

% numericalvalidation1 {k1 -> 2, k3 -> 2, b -> 4*Pi^2, k2 -> 1, k4 -> 5, p -> 1};
The asymptotic behavior of $\lambda_k$, given by Propositions \ref{asympexp2} and \ref{asympexp3}, can be numerically validated. \\
For instance in the case $k_1=k_3=2,\; k_2=1,\;k_4=5,\; b=4 \pi^2,$ (first case of Proposition \ref{asympexp3}), we have calculated numerically some large eigenvalues near the imaginary axis. From (\ref{cp1}) and (\ref{cp2})
it holds, in that case: $k^2\Re \lambda_k^j\sim \beta_j,\;j=1,2$, with

$$\beta_1=-\dfrac{10}{\pi ^2} \approx -0.202642,\;\;\;    \beta_2=-\dfrac{10}{\pi ^2} \approx -1.01321.$$

The table below confirms this behavior.

$$
\begin{array}{|c|c|c|c|c|c|}
\hline
k & 200 & 400 & 600 & 800 & 1000 \\
\hline
k^2 \Re \lb_k^1 & -0.202667 & -0.20265 & -0.202652 & -0.202588 & -0.202637\\
\hline
k^2 \Re \lb_k^2 & -1.01303 & -1.01317 & -1.0132 & -1.01324 & -1.01314\\
\hline
\end{array}
$$

\vspace{0.5cm}
The figures hereafter represent the eigenvalues of $\mathcal{A}$ in two cases: Figure 1 corresponds to Proposition \ref{asympexp2} and Figure 2 to the first case of Proposition \ref{asympexp3}.

%%%%%%%%%%%%%%%%%%%%%%%%%%%%%%%%%%%%%%%%%%%%%%%%%%%ù%%%%
%% Figure 1, figure avec {k1, k2, k3, k4, b} = {1, 2, 3, 2, 2};
%%%%%%%%%%%%%%%%%%%%%%%%%%%%%%%%%%%%%%%%%%%%%%%%%%%%%%%%%

\begin{figure}[h]
\begin{center}
\includegraphics [scale=0.5] {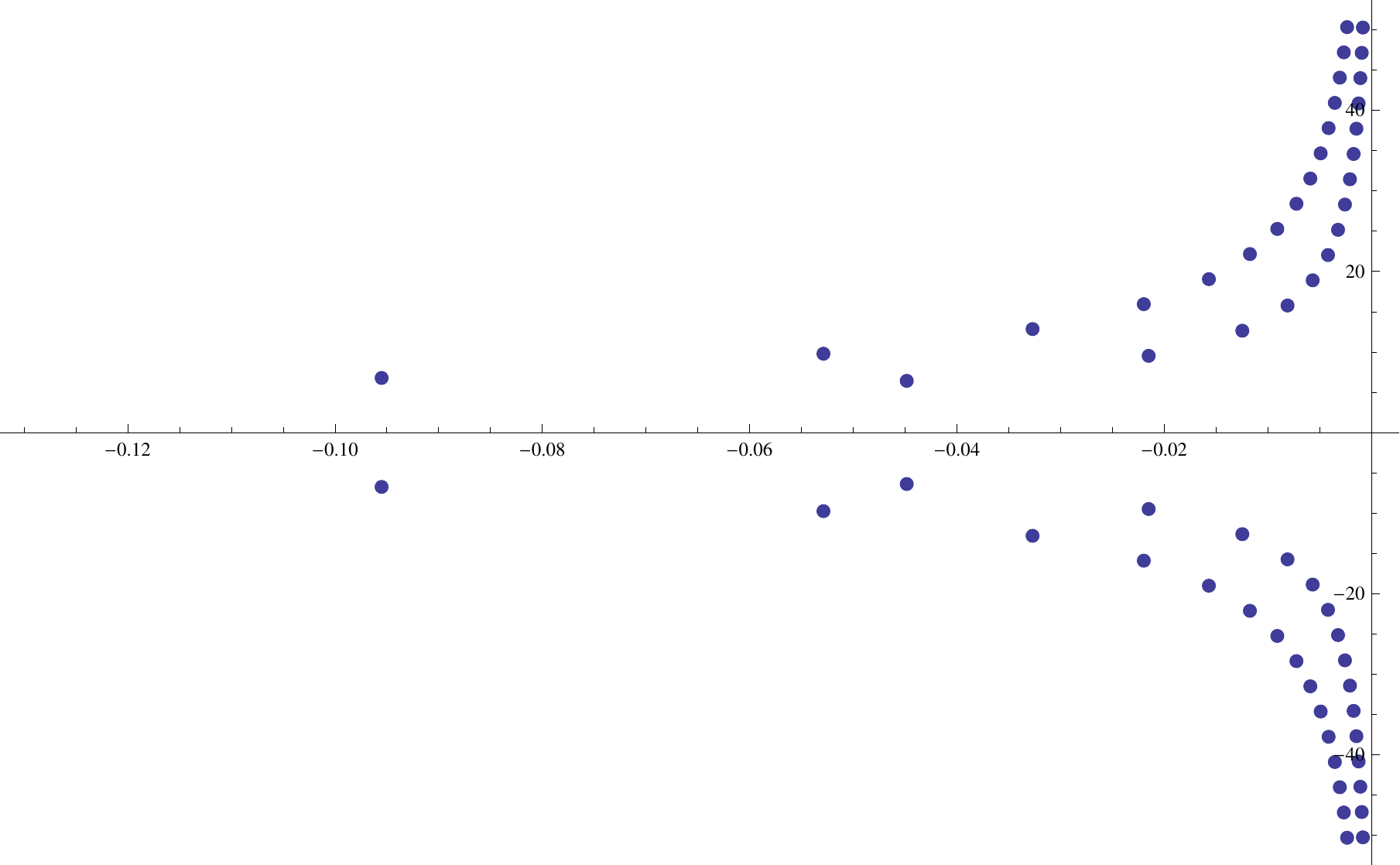}
\caption{  Eigenvalues of $\mathcal{A}$ if $a=1,b=2, k_1=1,k_2=2,k_3=3,k_4=2$ }
\label{figDN1}
\end{center}
\end{figure}

%%%%%%%%%%%%%%%%%%%%%%%%%%%%%%%%%%%%%%%%%%%%%%%%%%%ù%%%%
%% Figure 2, figure avec {k1, k2, k3, k4, b} = {2, 1, 2, 5, 4 \pi^2};
%%%%%%%%%%%%%%%%%%%%%%%%%%%%%%%%%%%%%%%%%%%%%%%%%%%%%%%%%

\begin{figure}[h]
\begin{center}
\includegraphics [scale=0.5] {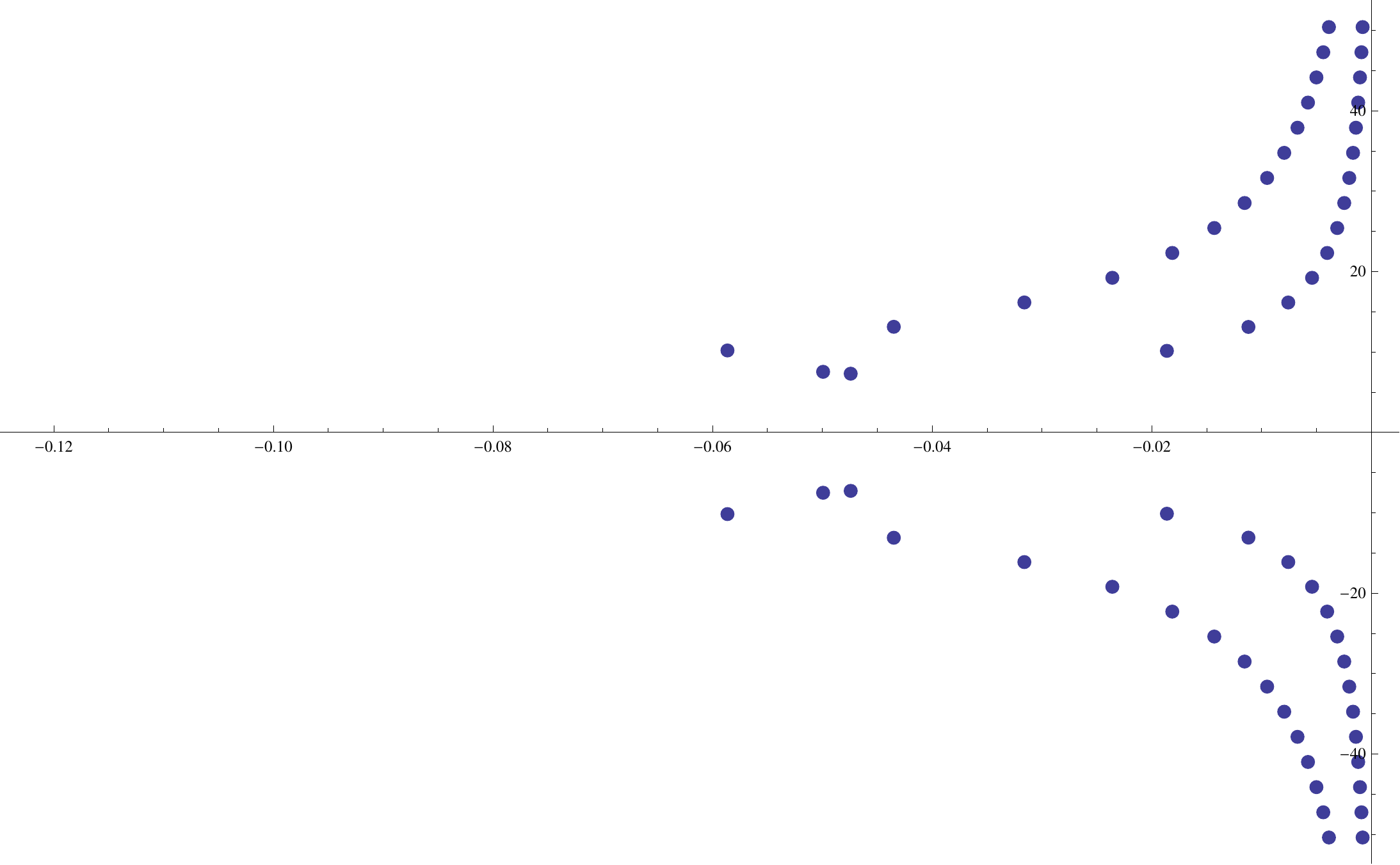}
\caption{Eigenvalues of $\mathcal{A}$ if $a = 1, b=4 \pi^2, k_1=k_3=2,\; k_2=1,\;k_4=5.$}
\label{fig2}
\end{center}
\end{figure}

\label{lastpage-01}
    
\edc